\documentclass[letterpaper]{article}

\usepackage{amsmath}
\usepackage{amsfonts}
\usepackage{amsthm}
\usepackage{amssymb}
\usepackage{datetime}
\usepackage{color}
\usepackage{url}
\usepackage{hyperref}

\newtheorem{theorem}{Theorem}
\newtheorem{lemma}[theorem]{Lemma}
\newtheorem*{lemma*}{Lemma}

\newtheorem{corollary}[theorem]{Corollary}
\newtheorem*{theorem*}{Theorem}
\theoremstyle{definition}
\newtheorem{definition}[theorem]{Definition}

\newtheorem{fact}[theorem]{Fact}

\newtheorem{remark}[theorem]{Remark}
\newtheorem*{remark*}{Remark}

\newcommand{\all}{\hbox{for all}}

\newcommand{\bra}[2]{\langle#1,#2\rangle}

\newcommand{\Bra}[2]{\big\langle#1,#2\big\rangle}
\newcommand{\B}{{\cal B}}

\newcommand{\dand}{\hbox{ and }}

\newcommand{\half}{{\textstyle\frac{1}{2}}}

\newcommand{\I}{\mathbb I}

\newcommand{\qlr}{\quad\Longrightarrow\quad}

\newcommand{\quand}{\quad\hbox{and}\quad}
\newcommand{\rbar}{\,]{-}\infty,\infty]}

\newcommand{\RR}{\mathbb R}

\newcommand{\toto}{\rightrightarrows}

\newcommand{\Cor}{Corollary~\ref}

\newcommand{\Def}{Definition~\ref}

\newcommand{\Fact}{Fact~\ref}
\newcommand{\Lem}{Lemma~\ref}

\newcommand{\Sec}{Section~\ref}

\newcommand{\Thm}{Theorem~\ref}

\title{Touching multifunctions on a Hilbert space}
\author{
Stephen Simons
\thanks{
Department of Mathematics, University of California, Santa Barbara, CA\ 93106-3080, U.S.A.
Email: \texttt{stesim38@gmail.com}.}}
\date{}

\begin{document}
\maketitle
\begin{abstract}
\noindent
We introduce the concept of the {\em touching} of two multifunctions on a real Hilbert space, and deduce that certain multifunctions on the space have a unique fixed point.   These result are applied to the theory of {\em genaralized cycles} and {\em generalized gap vectors} for the composition of the projections onto a finite number of closed convex space in a real Hilbert space.
\end{abstract}



%\pagestyle{myheadings}\markright{\rm\jobname, run on \today\ at \currenttime}
%\baselineskip20pt
% \Sec{Introduction}
\section{Introduction}\label{Introduction}
The analysis in this paper was originally motivated by \cite{ABRW}, in which the\break {\em geometry conjecture} (originally formulated in 1977, about the fixed point set of the composition of projections onto a finite number of nonempty closed convex subsets) was resolved.   The initial part of \cite{ABRW} contained some technical Hilbert space results, and the later part of \cite{ABRW} discussed some special cases and contained results on numerical computation.  \cite{mth} contained a streamlining of the Hilbert space results in \cite{ABRW}.   The techniques introduced in \cite{mth} were used in \cite{BW} to obtain further results on (classical and phantom) cycles and gap vectors.   This paper contains Hilbert space results that extend the main results in \cite{BW}.
\par
Here is a brief plan of this paper.   The analysis in \Sec{Nonlinear} is about multifunctions on a real Hilbert space, which we will denote by $Y$.   In \Def{touching}, we define the concepts of {\em touching multifunctions}, and give results in \Thm{MQthm} and \Cor{Q} on certain multifunctions that touch every maximally monotone multifunction.   \Cor{Q} leads rapidly to  \Lem{fpt} and \Cor{BW3.1}.   \Lem{fpt} is couched in terms of the {\em fixed points} of certain multifunctions on $Y$.  \Cor{BW3.1} is a restatement of \cite[Lemma 3.1]{BW} in the notation of this paper.      
\par
In \Sec{YX}, $Y$ is a closed subspace of a Hilbert space $X$. \Thm{BWTh4.10} is a (not altogether immediate) consequence of \Lem{fpt}. \Thm{BWTh4.10}(b) is a restatement of \cite[Theorem 4.10]{BW} in the notation of this paper and is generalized in \Thm{BWTh4.10}(a).        
%
% \Sec{Nonlinear}
\section{Very unmonotone multifunctions}\label{Nonlinear}
If $M$ is a multifunction, we write $G(M)$ for the {\em graph} of $M$.
%\Def{unmonotone}
\begin{definition}\label{unmonotone}
Let $\mu > 0$ and $Q\colon\ Y \toto Y$.   We say that $Q$ is {\em $\mu$--unmonotone} if
%\eqref{Very}
\begin{equation*}\label{Very}
(y_1,q_1),(y_2,q_2) \in G(Q) \qlr \bra{y_1 - y_2}{q_1-q_2} + \mu\|(y_1 - y_2,q_1 - q_2)\|^2 \le 0.
\end{equation*}
\end{definition}
We write $\B(Y)$ for the set of all {\em bounded linear operators} from $Y$ into $Y$, and  $\I_Y \in \B(Y)$ for the {\em identity map} on $Y$.
%\Lem{C-}
\begin{lemma}\label{C-}
Let $\mu > 0$ and $Q\colon\ Y \toto Y$ be $\mu$--unmonotone.   Then $-Q - \mu \I_Y$ is monotone.
\end{lemma}
\begin{proof}
For all $(y_1,q_1),(y_2,q_2) \in G(Q)$,
\begin{align*}
\bra{y_1 - y_2}{(-q_1 - \mu y_1) &- (-q_2 - \mu y_2)}
= \bra{y_1 - y_2}{-(q_1 - q_2) - \mu(y_1 - y_2)}\\
&= -\bra{y_1 - y_2}{q_1 - q_2} - \mu\|y_1 - y_2\|^2\\
&\ge -\bra{y_1 - y_2}{q_1 - q_2} - \mu\|(y_1 - y_2,q_1 - q_2)\|^2 \ge 0.
\end{align*}
This completes the proof of \Lem{C-}.   
\end{proof}
\Thm{MQthm} depends on two results from the theory of maximally monotone multifunctions, \Fact{SUMS} and \Fact{MintyThm}.   \Fact{SUMS}, follows from the {\em sum theorem},\break \cite[Corollary 24.4(i), p.\ 353]{BC}.
\par
If $M$ is a multifunction, we write $D(M)$ for the {\em domain} of $M$.
% \Fact{SUMS}
\begin{fact}\label{SUMS}
{\em Let $M_1$ and  $M_2$ be maximally monotone multifunctions on $Y$ and $D(M_1) = Y$.  Then $M_1 + M_2$ is maximally monotone.}

\end{fact}
\Fact{MintyThm}, is {\em Minty's theorem}, \cite{Minty}, or \cite[Theorem 21.1, pp.\ 311]{BC}.
% \Fact{MintyThm}
\begin{fact}\label{MintyThm}
{\em Let $N$ be a maximally monotone multifunction on $Y$ and $\mu > 0$. Then there exists $y \in Y$ such that $0 \in \mu y + Ny$.}
\end{fact}
%\Def{touching}
\begin{definition}\label{touching}
Let $M,Q\colon\ Y \toto Y$.   We say that $M$ and $Q$ {\em touch} if $G(M) \cap G(Q)$ is a singleton in $Y \times Y$.
\end{definition}
%\Thm{MQthm}
\begin{theorem}\label{MQthm}
Let $\mu > 0$, $Q\colon\ Y \toto Y$ be $\mu$--unmonotone and $D(Q) = Y$.   Suppose also that $-Q - \mu \I_Y$ is {\em maximally} monotone.  Then $Q$ touches every maximally monotone multifunction on $Y$.
\end{theorem}
\begin{proof}
Let $M\colon\ Y \toto Y$ be maximally monotone.   We start off by proving that $G(M) \cap G(Q)$ contains {\em at most one} element of $Y \times Y$.   To this end, let $(y_1,q_1),(y_2,q_2) \in G(M) \cap G(Q)$.   Since $(y_1,q_1),(y_2,q_2) \in G(M)$,
%\eqref{M}
\begin{equation}\label{M}
\bra{y_1 - y_2}{q_1 - q_2} \ge 0.
\end{equation}
On the other hand, since $(y_1,q_1),(y_2,q_2) \in G(Q)$, from \Def{unmonotone} and \eqref{M},
\begin{equation*}
\mu\|(y_1 - y_2,q_1 - q_2)\|^2 \le -\bra{y_1 - y_2}{q_1 - q_2} \le 0.
\end{equation*}
and so $(y_1,q_1) = (y_2,q_2)$.   Thus $G(M) \cap G(Q)$ contains at most one point.
\par
On the other hand, $M$ and $-Q - \mu \I_Y$ are both maximally monotone and $D(-Q - \mu \I_Y) = Y$.   From \Fact{SUMS}, $M -Q - \mu \I_Y$ is maximally monotone, and \Fact{MintyThm} provides $y \in Y$ such that $0 \in \mu y + (M -Q - \mu \I_Y)y$, that is to say, $0 \in (M - Q)y$.   It follows easily from this that $G(M) \cap G(Q) \ne \emptyset$, which completes the proof of \Thm{MQthm}.    
\end{proof}
We conclude this section with applications to bounded linear operators.   Our analysis depends on the following result about monotone functions.   See\break \cite[Corollary 20.25, p.\ 298]{BC}
%\Fact{single}
\begin{fact}\label{single}
{\em Any continuous (single--valued) monotone function from $Y$ into $Y$ is maximally monotone.}
\end{fact}
%\Cor{Q}
\begin{corollary}\label{Q}
Let $Q \in {\B}(Y)$, $\lambda > 0$ and,
%\eqref{le}
\begin{equation}\label{le}
\all\ y \in Y,\quad \bra{y}{Qy} + \lambda\|y\|^2 \le 0.
\end{equation}
Then $Q$ touches every maximally monotone multifunction on $Y$. 
\end{corollary}
\begin{proof}
Let $\mu := \lambda/(1 + \|Q\|^2)$.  Let $(y_1,q_1),(y_2,q_2) \in G(Q)$.   Then $q_1 = Qy_1$ and $q_2 = Qy_1$.   Consequently,  
\begin{align*}
\bra{y_1 - y_2}{q_1-q_2} &+ \mu\|(y_1 - y_2,q_1 - q_2)\|^2\\
&= \Bra{y_1 - y_2}{Q(y_1-y_2)} +\mu\|(y_1 - y_2,Q(y_1 - y_2)\|^2\\
&\le \Bra{y_1 - y_2}{Q(y_1-y_2)} + \mu(1 + \|Q\|^2)\|y_1 - y_2\|^2\\
&= \Bra{y_1 - y_2}{Q(y_1-y_2)} + \lambda\|y_1 - y_2\|^2 \le 0. 
\end{align*}
Thus $Q$ is $\mu$--unmonotone.   From \Lem{C-} and \Fact{single}, $-Q - \mu\I_Y$ is maximally monotone.  The result now follows from \Thm{MQthm}.
\end{proof}
%
%\Lem{fpt}
\begin{lemma}\label{fpt}
Let $T \in \B(Y)$ be surjective and bijective,  $\lambda > 0$ and
%\eqref{S}
\begin{equation}\label{S}
\all\ x \in Y,\ \bra{x}{Tx} + \lambda\|Tx\|^2 \le 0.
\end{equation}
Then, whenever $M$ is a maximally monotone multifunction on $Y$, the\break multifunction $MT$ has a unique fixed point.
\end{lemma}
\begin{proof}
From the open mapping theorem, there exists $Q \in \B(Y)$ such that
%\eqref{SUQ}
\begin{equation}\label{SUQ}
\all\ y \in Y,\quad y = T(Qy)\quand y = Q(Ty).
\end{equation}
Let $M\colon\ Y \toto Y$ be maximally monotone. 
If $y \in Y$, let $x := Qy \in Y$.  From \eqref{SUQ}, $y = T(Qy) = Tx$.   Thus, from \eqref{S}, $\bra{y}{Qy} + \lambda\|y\|^2 = \bra{Tx}{x} + \lambda\|Tx\|^2 \le 0$, and \Cor{Q} implies that $M$ and $Q$ touch.   Let $G(M) \cap G(Q) = \{(d,e)\} \subset Y \times Y$.
\par
If $y$ is a fixed point of $MT$ then $y \in M(Ty)$ and so $(Ty,y) \in G(M)$.   From \eqref{SUQ} again, $(Ty,y) \in G(Q)$, and so $(Ty,y) = (d,e)$, from which $y = e$.   Thus $e$ is the only possibility as a fixed point of $MT$.
\par
On the other hand, since $(d,e) \in G(M) \cap G(Q)$, $e = Qd$ and $e \in Md$.
From \eqref{SUQ}, $d = T(Qd) = Te$, and so $e \in M(Te) = (MT)e$.   Thus $e$ is, in fact, a fixed point of $MT$.
\end{proof}
The final result of this section, \Cor{BW3.1}, is a restatement of {\cite[Lemma 3.1]{BW}} which, in turn, generalizes {\cite[Lemma 16]{mth}}.   We will use \Fact{SUBDIFF} below, see\break \cite[Theorem 20.40, p.\ 304]{BC} and \cite[Theorem 21.2, p.\ 312]{BC} for proofs.
\par
If $H$ is a real Hilbert space, we write $\Gamma_0(H)$ for the set of all {\em proper, convex, lower semicontinuous} functions from $H$ into $\rbar$.
% \Fact{SUBDIFF}
\begin{fact}\label{SUBDIFF}
{\em Let $g \in \Gamma_0(Y)$.  Then the subdifferential of $g$, $\partial g$, is maximally\break monotone.}
\end{fact}
\par
The three quantities $X$, $Y$ and $Q$ used in \Cor{BW3.1} were defined in terms of three other quantities $R$, $M$ and $S$ in \cite[Section 2]{BW}.   The statement of\break \Cor{BW3.1} shows that this is unnecessary.   We also note that \eqref{eq} uses an equality, while \eqref{le} uses an inequality, which suffices for \Cor{Q}.
\par
We write $\B(X,Y)$ for the set of all bounded linear operators from $X$ into $Y$.  
%\Cor{BW3.1}
  
%\Cor{BW3.1}
\begin{corollary}\label{BW3.1}
Let $Y$ be a closed subspace of a Hilbert space $X$, $Q \in \B(X,Y)$,   
%\eqref{eq}
\begin{equation}\label{eq}
\all\ y \in Y,\quad \bra{y}{Qy} + \half\|y\|^2 = 0,
\end{equation}
$f \in \Gamma_0(X)$ and $f^*|_Y \in \Gamma_0(Y)$.   Then there exists $(d,e) \in Y \times Y$ such that
\begin{equation*}
G(\partial(f^*|_Y)) \cap G(Q)  = \{(d,e)\}.
\end{equation*}
\end{corollary}
\begin{proof}
This follows from \Fact{SUBDIFF}, \Cor{Q} and \Def{touching}, with the\break maximally monotone multifunction $\partial(f^*|_Y)$.  
\end{proof}
\begin{remark}
In \cite[Remark 3.6]{BW}, $e$  is called the {\em generalized cycle} of $f$, and $d$ is called the {\em generalized gap vector}  of $f$.
\par
A comparison of the statements of \Cor{Q} and \Cor{BW3.1} shows that $X$ does not play a fundamental role in \Cor{BW3.1}.
\par
On the other hand, $X$ {\em does} play a fundamental role in the results of the next section.
\end{remark}
% \Sec{YX}
\section{Hilbert subspaces}\label{YX}
\Thm{BWTh4.10}(b) is a restatement of {\cite[Theorem 4.10]{BW}} which, in turn, generalizes {\cite[Lemma 16]{mth}}.  \Thm{BWTh4.10}(b) follows easily from \Thm{BWTh4.10}(a). We note that \Thm{BWTh4.10}(a), does not require $f$ to attain a minimum on $X$, whereas \Thm{BWTh4.10}(b) does. 
%\Thm{BWTh4.10}
\begin{theorem}\label{BWTh4.10}
Let $Y$ be a closed subspace of a Hilbert space $X$, $S \in \B(X,Y)$, $S|_Y \in \B(Y)$ be surjective and injective and,
\begin{equation}\label{BW4.10.1}
\all\ x \in X,\ \bra{x}{Sx} + \half\|Sx\|^2 = 0.
\end{equation}
Let $f \in \Gamma_0(X)$ and $f^*|_Y \in \Gamma_0(Y)$.   Then it follows that the multifunction $[\partial(f^*|_Y)] \circ S|_Y\colon\ Y \toto Y$ has a unique fixed point, $e$, and
%\eqref{qlr}
\begin{equation}\label{qlr}
x \in X \dand Sx \in \partial f(x) \qlr Sx = Se.
\end{equation}
Furthermore:
\par
\noindent
{\rm (a)}\enspace  If $x \in X$, $f(x) \in \RR$ and $f(x) \le (f^*|_Y)^*(e)$ then
\begin{equation}\label{BW4.10.2}
Sx \in \partial f(x) \iff f^*(Sx) + \half\|Sx\|^2 + f(x) = 0\iff Sx = Se.
\end{equation} 
\par
\noindent
{\rm (b)}\enspace If $x \in X$ and $f(x) = \min_Xf$ then \eqref{BW4.10.2} holds.  
\end{theorem}
\begin{proof}
From \Fact{SUBDIFF}, $\partial(f^*|_Y)$ is maximally monotone, and so \Lem{fpt} \big(with $\lambda:= \half$, $T:= S|_Y$ and $M:= \partial(f^*|_Y)$\big) implies that $[\partial(f^*|_Y)]\circ S|_Y$ has a unique fixed point, $e$. Since $e \in [\partial(f^*|_Y)](S|_Ye)$,\quad $f^*(Se) = f^*|_Y(Se) < \infty$, $(f^*|_Y)^*(e) < \infty$\quad and, 
%\eqref{psipsi1}
\begin{equation}\label{psipsi1}
\bra{e}{Se} = f^*(Se) + (f^*|_Y)^*(e).
\end{equation}
We now establish \eqref{qlr}. To this end, suppose that $x \in X$ and $Sx \in \partial f(x)$.   Then  $f(x) < \infty$, $f^*(Sx) < \infty$ and
\begin{equation}
\bra{x}{Sx} = f(x) + f^*(Sx).
\end{equation}
From the Fenchel--Young inequality:
%\eqref{psipsi2}
\begin{equation}\label{psipsi2}
- \bra{x}{Se} \ge  -f(x) - f^*(Se) \quand - \bra{e}{Sx} \ge  - (f^*|_Y)^*(e) - f^*(Sx).
\end{equation}
Now\quad $\bra{x - e}{S(x - e)} = \bra{e}{Se} + \bra{x}{Sx} - \bra{x}{Se} - \bra{e}{Sx}$.\quad Consequently, by adding \eqref{psipsi1}--\eqref{psipsi2}, we see that \quad$\bra{x - e}{S(x - e)} \ge 0$.\quad  From \eqref{BW4.10.1},
\begin{equation*}
-\half\|Sx - Se\|^2 = -\half\|S(x - e)\|^2 = \bra{x - e}{S(x - e)} \ge 0.
\end{equation*}
Thus $Sx = Se$, which completes the proof of \eqref{qlr}.
\par
(a)\enspace Suppose that $x \in X$, $f(x) \in \RR$, $f(x) \le (f^*|_Y)^*(e)$ and $Sx = Se$.  Then, using \eqref{BW4.10.1} twice and \eqref{psipsi1},
\begin{align*}
\bra{x}{Sx} &= - \half\|Sx\|^2 = - \half\|Se\|^2= \bra{e}{Se} =  f^*(Sx) + (f^*|_Y)^*(e)\\
&\ge f^*(Sx) + f(x).
\end{align*}
Consequently, $Sx \in \partial f(x)$,  and \eqref{BW4.10.2} follows from \eqref{BW4.10.1} and \eqref{qlr}.
\par
(b)\enspace We first note that $(f^*|_Y)^*(e) \ge \bra{0}{e} - (f^*|_Y)(0) = -f^*(0) = \inf_Xf$.   So if $x \in X$, and $f(x) = \min_Xf$ then \eqref{BW4.10.2} follows from (a).
\end{proof}
\begin{remark}
The significance of \eqref{BW4.10.1} is not only that it makes \Thm{BWTh4.10} possible.   Since\quad $\bra{x}{Sx} + \half\|Sx\|^2 = 0 \iff \|x\|^2 + 2\bra{x}{Sx} + \|Sx\|^2 = \|x\|^2$\quad and\quad $\|x\|^2 + 2\bra{x}{Sx} + \|Sx\|^2 = \|Sx + x\|^2$,\quad \eqref{BW4.10.1} is equivalent to the statement that the linear map $S + \I_X$ is an {\em isometry}.   Historically, this is backwards: in \cite{mth} and \cite{BW}, the parameters of the {\em geometry conjecture} provided us with an isometry $R$, and $S$ was {\em defined} by $S:= R - \I_X$.
\end{remark}

\end{document}